\documentclass{amsart}
\usepackage{amsmath,amssymb, amsthm}
\newtheorem{theorem}{Theorem}[section]
\newtheorem{lemma}[theorem]{Lemma}
\newtheorem{proposition}[theorem]{Proposition}
\newtheorem{corollary}[theorem]{Corollary}

\newtheorem{example}[theorem]{Example}

\begin{document}

\title[a maximal immediate extension of finite rank]{Valuation domains with a maximal immediate extension of finite rank.}
\author{Fran\c{c}ois Couchot}
\address{Laboratoire de Math\'ematiques Nicolas Oresme, CNRS UMR
  6139,
D\'epartement de math\'ematiques et m\'ecanique,
14032 Caen cedex, France}
\email{couchot@math.unicaen.fr}



\begin{abstract}
If $R$ is a valuation domain of maximal ideal $P$ with a maximal immediate extension of finite rank it is proven that there exists a finite sequence of prime ideals $P=L_0\supset L_1\supset\dots\supset L_m\supseteq 0$ such that $R_{L_j}/L_{j+1}$ is almost maximal for each $j$, $0\leq j\leq m-1$ and $R_{L_m}$ is maximal if $L_m\ne 0$. Then we suppose that there is an integer $n\geq 1$ such that each torsion-free $R$-module of finite rank is a direct sum of modules of rank at most $n$. By adapting Lady's methods, it is shown that $n\leq 3$ if $R$ is almost maximal, and  the converse holds if $R$ has a maximal immediate extension of rank $\leq 2$.
\end{abstract}

\maketitle


Let $R$ be a valuation domain of maximal ideal $P$, $\widehat{R}$ a maximal immediate extension of $R$, $\widetilde{R}$ the completion of $R$ in the $R$-topology, and $Q$, $\widehat{Q}$, $\widetilde{Q}$ their respective fields of quotients. If $L$ is a prime ideal of $R$, as in \cite{FaZa86}, we define the \textbf{total defect} at $L$, $d_R(L)$, the \textbf{completion defect} at $L$, $c_R(L)$, as the rank of the torsion-free $R/L$-module $\widehat{(R/L)}$ and the rank of the torsion-free $R/L$-module $\widetilde{(R/L)}$, respectively. Recall that a local ring $R$ is \textbf{Henselian} if each indecomposable module-finite $R$-algebra is local and a valuation domain is \textbf{strongly discrete} if it has no non-zero idempotent prime ideal. The aim of this paper is to study  valuation domains $R$ for which $d_R(0)<\infty$. The first example of a such valuation domain was given by Nagata \cite{Nag62}; it is a Henselian rank-one discrete valuation domain of characteristic $p>0$ for which $d_R(0)=p$. By using a generalization of Nagata's idea, Facchini and Zanardo gave other examples of characteristic $p>0$, which are Henselian and strongly discrete. More precisely:
\begin{example}
\label{E:1}\cite[Example 6]{FaZa86}
For each prime integer $p$ and for each finite sequence of integers $\ell(0)=1,\ \ell(1),\dots,\ell(m)$ there exists a strongly discrete valuation domain $R$ with prime ideals $P=L_0\supset L_1\ \supset\dots\supset L_m=0$ such that $c_R(L_i)=p^{\ell(i)},$ $\forall i,\ 1\leq i\leq m$.
\end{example}
So, $d_R(0)=p^{(\sum_{i=1}^{i=m}\ell(i))}$ by \cite[Corollary 4]{FaZa86}.
\begin{theorem}
\label{E:2}\cite[Theorem 8]{FaZa86}
Let $\alpha$ be an ordinal number, $\ell:\alpha+1\rightarrow \mathbb{N}\cup\{\infty\}$ a mapping with $\ell(0)=1$ and $p$ a prime integer. Then there exists a strongly discrete valuation domain $R$ and an antiisomorphism $\alpha+1\rightarrow\mathrm{Spec}(R)$, $\lambda\mapsto L_{\lambda},$ such that $c_R(L_{\lambda})=p^{\ell(\lambda)}$, $\forall\lambda\leq\alpha$.
\end{theorem}
 So, if $\ell(\lambda)=0$, $\forall \lambda\leq \alpha$, except for a finite subset, then $d_R(0)<\infty$ by \cite[Corollary 4]{FaZa86}.
 
In 1990, V\'amos gave a complete characterization of non-Henselian valuation domains with a finite total defect, and  examples of such rings. His results are summarized in the following: 
\begin{theorem}
\label{T:vam}\cite[Theorem 5]{Vam90}
 Let $R$ be a non-Henselian valuation domain and assume that $d_R(0)<\infty$. Then one of the following holds:
\begin{itemize}
\item[{\rm (rc)}]  $d_R(0)=2$, $R$ has characteristic zero, $\widehat{Q}$ is algebraically closed and its cardinality $\vert\widehat{Q}\vert^{\aleph_0}=\vert\widehat{Q}\vert$. Further, $Q$ is real-closed, the valuation on $Q$ has exactly two extensions to $\widehat{Q}$ and $R$ is almost maximal.
\item[{\rm (y)}] There is a non-zero prime ideal $L$ of $R$ such that $R_L$ is a maximal valuation ring, and $R/P$ and its field of quotients satisfy {\rm (rc)}.
\end{itemize} 
\end{theorem}

As V\'amos, if $R$ is a domain, we say that $fr(R)\leq n$ (respectively $fr^o(R)\leq n$) if every torsion-free module (respectively every submodule of a free module) of finite rank is a direct sum of modules of rank at most $n$. By \cite[Theorem 3]{Vam90} $fr(R)\geq d_R(0)$ if $R$ is a valuation domain. So, the study of valuation domains $R$ for which $d_R(0)<\infty$ is motivated by the problem of the characterization of valuation domains $R$ for which $fr(R)<\infty$. 

When $R$ is a valuation domain which is a $\mathbb{Q}$-algebra or not Henselian, then $fr(R)<\infty$ if and only if $fr(R)=d_R(0)\leq 2$ by \cite[Theorem 10]{Vam90}. Moreover, if $fr(R)=2$, either  $R$ is of type (rc)  and $fr^o(R)=1$ or $R$ is of type (y) and $fr^o(R)=2$. When $R$ is a rank-one discrete valuation domain, then $fr(R)<\infty$ if and only if $fr(R)=d_R(0)\leq 3$ by \cite[Theorem 8]{Zan92} and \cite[Theorem 2.6]{ArDu93}.

In this paper we complete V\`amos's results. In Section~\ref{S:MIE}, a description of valuation domains with a finite total defect is given by Theorem~\ref{T:main} and Proposition~\ref{P:factors}. In Section~\ref{S:torsion} we give some precisions on the structure of torsion-free $R$-modules of finite rank when $R$ satisfies a condition weaker than $d_R(0)<\infty$. In Section~\ref{S:bound} we extend to every almost maximal valuation domain the methods used by Lady in \cite{Lad77} to study  torsion-free  modules over rank-one discrete valuation domains. If $R$ is an almost maximal valuation domain, we prove that $d_R(0)\leq 3$ if $fr(R)<\infty$ and that $fr(R)=d_R(0)$ if $d_R(0)\leq 2$.

For definitions and general facts about valuation rings and their 
modules we refer to the books by  Fuchs and  Salce \cite{FuSa85} and \cite{FuSa01}.

\section{maximal immediate extension of finite rank}
\label{S:MIE}
We recall some preliminary results  needed to prove Theorem~\ref{T:main} which gives a description of valuation domains with a finite total defect.

Let $M$ be a non-zero module over a valuation domain $R$. As in
\cite[p.338]{FuSa01} we set
$M^{\sharp}=\{s\in R\mid sM\subset M\}.$ 
Then $M^{\sharp}$ is a prime ideal of $R$ and is called the \textbf{top prime ideal} associated with $M$. 

\begin{proposition}
\label{P:localcompl} Let $A$ be a proper ideal of $R$ and let $L$ be a prime ideal such that $A^{\sharp}\subseteq L$ and $A$ is not isomorphic to $L$. Then $R/A$ is complete in its ideal topology if and only if $R_L/A$ is also complete in its ideal topology.
\end{proposition}
\begin{proof}  Let $(a_i+A_i)_{i\in I}$ be a family of cosets of $R_L$ such that $a_i\in a_j+A_j$ if $A_i\subset A_j$ and such that $A=\cap_{i\in I}A_i$. We may assume that $A_i\subseteq L,\ \forall i\in I$. So, $a_i+L=a_j+L,\ \forall i,j\in I$. Let $b\in a_i+L,\ \forall i\in I$. It follows that $a_i-b\in L,\ \forall i\in I$. If $R/A$ is complete in the $R/A$-topology, $\exists c\in R$ such that $c+b-a_i\in A_i,\ \forall i\in I$. Hence $R_L/A$ is complete in the $R_L/A$-topology too.

Conversely let $(a_i+A_i)_{i\in I}$ be a family of cosets of $R$ such that $a_i\in a_j+A_j$ if $A_i\subset A_j$ and such that $A=\cap_{i\in I}A_i$. We may assume that $A\subset A_i\subseteq L,\ \forall i\in I$. We put $A_i'=(A_i)_L,\ \forall i\in I$. We know that $A=\cap_{a\notin A}La$. Consequently, if $a\notin A$, there exists $i\in I$ such that  $A_i\subseteq La$, whence $A_i'\subseteq La$. It follows that $A=\cap_{i\in I}A_i'$. Clearly, $a_i\in a_j+A_j'$ if $A_i'\subset A_j'$. Then there exists $c\in R_L$ such that $c\in a_i+A_i',\ \forall i\in I$. Since $A_i'\subset R,\ \forall i\in I$, $c\in R$. From $A=\cap_{j\in I}A_j'$ and $A\subset A_i,\ \forall i\in I$ we deduce that $\forall i\in I,\ \exists j\in I$ such that $A_j'\subset A_i$. We get that $c\in a_i+A_i$ because $c-a_j\in A_j'\subseteq A_i$ and $a_j-a_i\in A_i$. So, $R/A$ is complete in the $R/A$-topology. \end{proof}

\begin{proposition}
\label{P:prime} \cite[Exercise II.6.4]{FuSa01} Let $R$ be a valuation ring and let $L$ be a non-zero prime ideal. Then $R$ is (almost) maximal if and only if $R/L$ is maximal and $R_L$ is (almost) maximal.
\end{proposition}
\begin{proof} If $R$ is (almost) maximal, it is obvious that $R/L$ is maximal and by Proposition~\ref{P:localcompl} $R_L$ is (almost) maximal. Conversely let $A$ be a non-zero ideal and $J=A^{\sharp}$. Suppose that either $J\subset L$ or $J=L$ and $A$ is not isomorphic to $L$. Since $R_L$ is (almost) maximal it follows that $R_L/A$ is complete in its ideal topology. From Proposition~\ref{P:localcompl} we deduce that $R/A$ is complete in its ideal topology. Now, suppose that $L\subset J$. If $A\subset L$ let $t\in J\setminus L$. Thus $A\subset t^{-1}A$. Let $s\in t^{-1}A\setminus A$. Therefore $L\subset tR\subseteq s^{-1}A$. So, $R/s^{-1}A$ is complete in its ideal topology because $R/L$ is maximal, whence $R/A$ is complete too. Finally if $A\cong L$ the result is obvious. \end{proof}

\begin{proposition}
\label{P:primeunion} Let $(L_{\lambda})_{\lambda\in\Lambda}$ be a non-empty family of prime ideals of $R$ and let $L=\cup_{\lambda\in\Lambda}L_{\lambda}$. Then $L$ is prime, $R_L=\cap_{\lambda\in\Lambda}R_{L_{\lambda}}$ and $R_L$ is maximal if and only if $R_{L_{\lambda}}$ is maximal $\forall\lambda\in\Lambda$.
\end{proposition}
\begin{proof} It is obvious that $L$ is prime.  Let $Q$ be the field of fractions of $R$.  If \ $x\in Q\setminus R_L$ \ then \ $x = \frac{1}{s}$ \ where \ $s\in L$.  Since \ $L=\cup_{\lambda\in\Lambda}L_{\lambda}$, \
$\exists \mu\in\Lambda$ \ such that \ $s\in L_{\mu}$.  We deduce that \ $x\notin 
R_{L_{\mu}}$ \ and \
$R_L=\cap_{\lambda\in\Lambda}R_{L_{\lambda}}$.  

If $R_L$ is maximal, we deduce that $R_{L_{\lambda}}$ is maximal $\forall\lambda\in\Lambda$ by Proposition~\ref{P:localcompl}. Conversely, by \cite[proposition 4]{Zel53} $R_L$ \ is
linearly compact in the inverse limit topology.  Since \ $R_L$ \ is  
Hausdorff 
in this linear topology then every nonzero ideal is open and also 
closed. Hence \
$R_L$ \ is linearly compact in the discrete topology. \end{proof}

\bigskip Recall that a valuation domain $R$ is \textbf{Archimedean} if its maximal $P$ is the only non-zero prime ideal and an ideal $A$ is \textbf{Archimedean} if $A^{\sharp}=P$.

\begin{proposition}\cite[Corollary 9]{Cou05} \label{P:Archimedean} Let $R$ be an Archimedean valuation domain. If $d_R(0) <\infty$, then $R$ is almost maximal. 
\end{proposition}

\bigskip
From Propositions~\ref{P:localcompl}, \ref{P:prime}, \ref{P:primeunion} and \ref{P:Archimedean} we deduce the following:

\begin{proposition}
\label{P:almost} Let $R$ be a valuation domain such that $d_R(0)<\infty$ and  $R/A$ is Hausdorff and complete in its ideal topology for each non-zero non-Archimedean ideal $A$.  Then $R$ is almost maximal.
\end{proposition}
\begin{proof} Let $L,\ L'$ be prime ideals such that $L'\subset L$. Since $\widehat{(R_L)}$ is a summand of $(\widehat{R})_L$ we have $d_{R_L}(0)\leq d_R(0)$. On the other hand, by tensoring a pure-composition series of $\widehat{(R_L)}$ with $R_L/L'$ we get a pure-composition series of $\widehat{(R_L/L')}$. So, $d_{R_L}(L')\leq d_R(0)$.

If $R$ is Archimedean the result follows from Proposition~\ref{P:Archimedean}. Suppose that $R$ is not Archimedean, let $J$ be a non-zero ideal and let $(L_{\lambda})_{\lambda\in\Lambda}$ be the family of prime ideals properly containing $J$ and properly contained in $P$. If $\Lambda=\emptyset$ we get that $R$ is almost maximal by applying Propositions~\ref{P:Archimedean} and \ref{P:prime}. Else, let $L'=\cup_{\lambda\in\Lambda}L_{\lambda}$. By Proposition~~\ref{P:primeunion}, $R_{L'}/J$ is maximal for each non-zero prime $J$. If $L'\ne P$ then $R/L'$ is maximal by Proposition~\ref{P:Archimedean} and it follows that $R/J$ is maximal by Proposition~\ref{P:prime}. If the intersection $K$ of all non-zero primes is zero then $R$ is almost maximal. If $K\ne 0$ then $R_K$ is Archimedean. We conclude by using Propositions~\ref{P:Archimedean} and \ref{P:prime}. \end{proof}

\bigskip
Given a ring $R$, an $R$-module $M$ and $x\in M$,  the \textbf{content ideal} $\mathrm{c}(x)$ of $x$ in $M$, is the intersection of all ideals $A$ for which $x\in AM$. We say that $M$ is a \textbf{content module} if $x\in\mathrm{c}(x)M,\ \forall x\in M$.

\begin{lemma}
\label{L:cont} Let $U$ be a torsion-free module such that $U=PU$. Then:
\begin{enumerate}
\item $\forall x\in U,\ x\ne 0,\ x\notin\mathrm{c}(x)U$;
\item let $0\ne x,y\in U$ and $t\in R$ such that $x=ty$. Then $\mathrm{c}(y)=t^{-1}\mathrm{c}(x)$;
\item if $U$ is uniserial then, for each $x\in U,\ x\ne 0,\ \mathrm{c}(x)^{\sharp}=U^{\sharp}.$
\end{enumerate}   
\end{lemma}
\begin{proof} $(1)$. If $x\in\mathrm{c}(x)U$, there exist $a\in R$ and $z\in U$ such that $x=az$ and $\mathrm{c}(x)=Ra$. But, since $z\in PU$, we get a contradiction. 

$(2)$. Let $0\ne x,y\in U$ such that $x=ty$. If $s\notin\mathrm{c}(y)$ then $x=tsz$ for some $z\in U$ and $st\notin\mathrm{c}(x)$. So, $s\notin t^{-1} \mathrm{c}(x)$. Conversely, if $s\notin t^{-1}\mathrm{c}(x)$ then $st\notin\mathrm{c}(x)$. We have $x=stz$ for some $z\in U$. We get that $y=sz$. So, $s\notin\mathrm{c}(y)$.

$(3)$. We put $A=\mathrm{c}(x)$ and $L=A^{\sharp}$. Let $s\notin L$ and $y\in U$ such that $x=ty$ for some $t\in R$. Then $\mathrm{c}(y)=t^{-1}A$ and $t\notin A$. So, $t^{-1}A\subseteq L$. Consequently $y\in sU$. Let $s\in L$. If $s\in A$ then $x\notin sU$. If $s\in L\setminus A$ let $t\in s^{-1}A\setminus A$. There exists $y\in U$ such that $x=ty$. Since $\mathrm{c}(y)=t^{-1}A$ and $s\in t^{-1}A$ we deduce that $y\notin sU$. \end{proof}

\bigskip
This lemma and the previous proposition allow us to show the following theorem.

\begin{theorem}
\label{T:main} Let $R$ be a valuation domain such that $d_R(0)<\infty$. Then there exists a finite family of prime ideals $P=L_0\supset L_1\supset\dots\supset L_{m-1}\supset L_m\supseteq 0$ such that $R_{L_k}/L_{k+1}$ is almost maximal, $\forall k,\ 0\leq k\leq m-1$ and $R_{L_m}$ is maximal if $L_m\ne 0$ (or equivalently, for each proper ideal $A\ncong L_k,\ \forall k,\ 0\leq k\leq m$, $R/A$ is Hausdorff and complete in its ideal topology). Moreover, $d_R(0)=\prod_{k=1}^{k=m}c_R(L_k)$. 
\end{theorem}
\begin{proof} Let $n=d_R(0)$. Then $\widehat{R}$ has a pure-composition series
\[0=G_0\subset R=G_1\subset\dots\subset G_{n-1}\subset G_n=\widehat{R}\]
such that, $\forall k,\ 1\leq k\leq n,\ U_k=G_k/G_{k-1}$ is a uniserial torsion-free module. The family $(L_0,\dots,L_m)$ is defined in the following way: $\forall j,\ 0\leq j\leq m$, there exists $k,\ 1\leq k\leq n$ such that $L_j=U^{\sharp}_k$.

Now, let $A$ be a proper ideal such that $R/A$ is Hausdorff and non-complete in its ideal topology. By \cite[Lemma V.6.1]{FuSa01} there exists $x\in\widehat{R}\setminus R$ such that $A=\mathrm{c}(x+R)$ (Clearly $\mathrm{c}(x+R)=\mathrm{B}(x)$, the breadth ideal of $x$). Let $U$ be a pure uniserial submodule of $\widehat{R}/R$ containing $x+R$ and let $M$ be the inverse image of $U$ by the natural map $\widehat{R}\rightarrow\widehat{R}/R$. From the pure-composition series of $M$ with factors $R$ and $U$, and a pure-composition series of $\widehat{R}/M$ we get a pure-composition series of $\widehat{R}$. Since each pure composition series has isomorphic uniserial factors by \cite[Theorem XV.1.7]{FuSa01}, it follows that $U\cong U_k$ for some $k,\ 2\leq k\leq n$. So, by Lemma~\ref{L:cont} $A^{\sharp}=U^{\sharp}=U_k^{\sharp}$. 

 We apply Proposition~\ref{P:almost} and deduce that $R_{L_k}/L_{k+1}$ is almost maximal $\forall k,\ 0\leq k\leq m-1$ and $R_{L_m}$ is maximal if $L_m\ne 0$.

To prove the last assertion we apply \cite[Lemma 2]{FaZa86} (The conclusion of this lemma holds if $R_L/L'$ is almost maximal, where $L$ and $L'$ are prime ideals, $L'\subset L$). \end{proof}

\bigskip
The following completes the previous theorem.

\begin{proposition}
\label{P:factors} Let $R$ be a valuation domain such that $d_R(0)<\infty$, let $(U_k)_{1\leq k\leq n}$ be the family of uniserial factors of all pure-composition series of $\widehat{R}$ and let\\
 $(L_j)_{0\leq j\leq m}$ be the family of prime ideals defined in Theorem~\ref{T:main}. Then:
\begin{enumerate}
\item $\forall k,\ 1\leq k\leq n$, $U_k\cong R_{U_k^{\sharp}}$;
\item  $\widehat{R}$ has a pure-composition series
\[0=F_0\subset R=F_1\subset\dots\subset F_{m-1}\subset F_m=\widehat{R}\]
where $F_{j+1}/F_j$ is a free $R_{L_j}$-module of finite rank, $\forall j,\ 0\leq j\leq m-1$.
\end{enumerate}  
\end{proposition}
\begin{proof} $(1)$. Let $A$ be an ideal such that $\exists j,\ 0\leq j\leq m,\ A^{\sharp}=L_j$ and $A\ncong L_j$. In the sequel we put $L_{m+1}=0$ if $L_m\ne 0$.
 
 First, for each uniserial torsion-free module $U$, we will show that each family $(x_r+rU)_{r\in R\setminus A}$ has a non-empty intersection if $x_r\in x_t+tU$, $\forall r,t\in R\setminus A,\ r\in tR$. As in the proof of Proposition~\ref{P:prime} we may assume that $L_{j+1}\subset A$. Since $R_{L_j}/L_{j+1}$ is almost maximal and $A$ is an ideal of $R_{L_j}$ the family $(x_r+rU_{L_j})_{r\in R\setminus A}$ has a non-empty intersection. If $r\in L_j\setminus A$, we have $r^{-1}A\subset L_j$. So, if $t\in L_j\setminus r^{-1}A$ then $rt\notin A$ and $rtU_{L_j}\subseteq rU$. It follows that we can  do as in the proof of Proposition~\ref{P:localcompl} to show that the family $(x_r+rU)_{r\in R\setminus A}$ has a non-empty intersection.
 
 Let \(0=G_0\subset R=G_1\subset\dots\subset G_{n-1}\subset G_n=\widehat{R}\) be a pure-composition series of $\widehat{R}$ whose factors are the $U_k,\ 1\leq k\leq n$. By induction on $k$ and by using the pure-exact sequence \(0\rightarrow G_{k-1}\rightarrow G_k\rightarrow U_k\rightarrow 0\),  we get that each family $(x_r+rG_k)_{r\in R\setminus A}$ for which $x_r\in x_t+tG_k,\ \forall r,t\in R\setminus A,\ r\in tR$, has a non-empty intersection.
 
 Let $k,\ 2\leq k\leq n,$ be an integer, let $0\ne x\in G_k\setminus G_{k-1}$ and let $A=\mathrm{c}(x+G_{k-1})$. Then $A^{\sharp}=U_k^{\sharp}=L_j$ for some $j,\ 1\leq j\leq m$. We shall prove that $A\cong L_j$. For each $r\in R\setminus A,\ x=g_r+ry_r$ for some $g_r\in G_{k-1}$ and $y_r\in G_k$. Let $r,t\in R\setminus A$ such that $r\in tR$. Then we get that $g_r\in g_t+tG_k\cap G_{k-1}=g_t+tG_{k-1}$ since $G_{k-1}$ is a pure submodule. If $A\ncong L_j$ the family $(g_r+rG_{k-1})_{r\in R\setminus A}$ has a non-empty intersection. Let $g\in g_r+rG_{k-1}\ \forall r\in R\setminus A$. Then $(x-g)\in rG_k,\ \forall r\in R\setminus A$. Since $G_k$ is a pure-essential extension of a free module, $G_k$ is a content module by \cite[Proposition 23]{Couch07}. It follows that $(x-g)\in AG_k$. So $x+G_{k-1}\in AU_k$. But, since $k\geq 2$, $U_k=PU_k$ because $\widehat{R}/P\widehat{R}\cong R/PR$. So, $x+G_{k-1}\notin AU_k$. From this contradiction we get that $A=sL_j$ for some $0\ne s\in R$. If $sL_j\ne L_j$ then $x+G_{k-1}=sy+G_{k-1}$ for some $y\in G_k$ because $s\notin A$. If follows that $\mathrm{c}(y+G_{k-1})=L_j$. We put $y'=y+G_{k-1}$. Then, for each $z\in U_k\setminus Ry'$ there exists $t\in R\setminus L_j$ such that $y'=tz$. We get that $U_k=R_{L_j}y'$. 
 
 $(2)$. Let $M=\widehat{R}/R$. Then $L_1=M^{\sharp}$. From above we get that $M/L_1M\ne 0$. By \cite[Proposition 21]{Couch07} $M$ contains a pure free $R_{L_1}$-submodule $N$ such that $N/L_1N\cong M/L_1M$. It follows that $(M/N)^{\sharp}=L_2$. We set $F_2$  the inverse image of $N$ by the natural map $\widehat{R}\rightarrow M$. We complete the proof by induction on $j$. \end{proof}

\section{Torsion-free modules of finite rank.}
\label{S:torsion}

In this section we give some precisions on the structure of torsion-free $R$-modules of finite rank when $R$ satisfies a condition weaker than $d_R(0)<\infty$. The following lemmas are needed.

\begin{lemma}
\label{L:cycl} Let $R$ be a valuation ring (possibly with zerodivisors), let $U$  be a uniserial module and let $L$ be a prime ideal such that $L\subset U^{\sharp}$. Then $U_L$ is a cyclic $R_L$-module.
\end{lemma}
\begin{proof} Let $s\in U^{\sharp}\setminus L$ and let $x\in U\setminus sU$. Let $y\in U\setminus Rx$. There exists $t\in P$ such that $x=ty$. Then $t\notin Rs$, whence $t\notin L$. It follows that $U_L=R_Lx$. \end{proof}

\begin{lemma}
\label{L:ext1} Let $R$ be a valuation ring (possibly with zerodivisors), let $U$ and $V$ be uniserial modules such that $V^{\sharp}\subset U^{\sharp}$. Assume that $U_L$ is a faithful $R_L$-module, where $L=V^{\sharp}$. Then $\mathrm{Ext}_R^1(U,V)=0$.
\end{lemma}
\begin{proof} Let $M$ be an extension of $V$ by $U$. By lemma~\ref{L:cycl} $U_L$ is a free cyclic $R_L$-module. Since $V$ is a module over $R_L$, it follows that $V$ is a summand of $M_L$. We deduce that $V$ is a summand of $M$ too. \end{proof}

\begin{lemma}
\label{L:ext2} Let $R$ be a valuation domain for which there exists a prime ideal $L\ne P$ such that $R/L$ is almost maximal. Then $\mathrm{Ext}_R^1(U,V)=0$ for each pair of ideals  $U$ and $V$  such that $L\subset U^{\sharp}\cap V^{\sharp}$.
\end{lemma}
\begin{proof} Let $M$ be an extension of $V$ by $U$. It is easy to check that $U/LU$ and $V/LV$ are non-zero and non divisible $R/L$-modules. Since $R/L$ is almost maximal $M/LM\cong U/LU\oplus V/LV$ by \cite[Proposition VI.5.4]{FuSa85}. If $L\ne 0$, it follows that there exist two submodules $H_1$ and $H_2$ of $M$, containing $LM$, such that $H_1/LM\cong U/LU$ and $H_2/LM\cong V/LV$. For $i=1,2$ let $x_i\in H_i\setminus LM$ and let $A_i$ be the submodule of $H_i$ such that $A_i/Rx_i$ is the torsion submodule of $H_i/Rx_i$. Then $A_i+LM/LM$ is a non-zero pure submodule of $H_i/LM$ which is of rank one over $R/L$. It follows that $H_i=A_i+LM$. By Lemma~\ref{L:cycl} $M_L\cong V_L\oplus U_L$. We deduce that $LM\cong LM_L$ is a direct sum of uniserial modules. Since $A_i\cap LM$ is a non-zero pure submodule of $LM$ there exists a submodule $C_i$ of $LM$ such that $LM=(A_i\cap LM)\oplus C_i$ by \cite[Theorem XII.2.2]{FuSa01}. It is easy to check that $H_i=A_i\oplus C_i$. From $M=H_1+H_2$ and $LM=H_1\cap H_2$ we deduce that the following sequence is pure exact: \[0\rightarrow LM\rightarrow H_1\oplus H_2\rightarrow M\rightarrow 0,\]
where the homomorphism from $LM$ is given by $x\mapsto(x,-x),\ x\in LM$, and the one onto $M$ by $(x,y)\mapsto x+y,\ x\in H_1,\ y\in H_2$.
Since $H_1\oplus H_2$ is a direct sum of uniserial modules, so is $M$ by \cite[Theorem XII.2.2]{FuSa01}. Consequently $M\cong V\oplus U$. \end{proof}

\begin{proposition}
\label{P:torsion-free1} Let $R$ be a valuation domain. Let $G$ be a torsion-free $R$-module of finite rank. Then $G$ has a pure-composition series with uniserial factors $(U_k)_{1\leq k\leq n}$ such that $U_k^{\sharp}\supseteq U_{k+1}^{\sharp},\ \forall k,\ 1\leq k\leq n-1$.
\end{proposition}
\begin{proof} $G$ has a pure-composition series
\[0=H_0\subset H_1\subset\dots\subset H_{n-1}\subset H_n=G\]
such that, $\forall k,\ 1\leq k\leq n,\ V_k=H_k/H_{k-1}$ is a uniserial torsion-free module. Suppose there exists $k,\ 1\leq k\leq (n-1)$ such that $V_k^{\sharp}\subset V_{k+1}^{\sharp}$. By Lemma~\ref{L:ext1}, $H_{k+1}/H_{k-1}\cong V_k\oplus V_{k+1}$. So, if $H_k'$ is the inverse image of $V_{k+1}$ by the surjection $H_{k+1}\rightarrow H_{k+1}/H_{k-1}$, if $U_k=H_k'/H_{k-1}$ and $U_{k+1}=H_{k+1}/H_k'$ then $U_k^{\sharp}\supset U_{k+1}^{\sharp}$. So, in a finite number of similar steps, we get a pure composition series with the required property. \end{proof}

\begin{proposition}
\label{P:torsion-free} Let $R$ be a valuation domain. Assume that there exists a finite family of prime ideals $P=L_0\supset L_1\supset\dots\supset L_{m-1}\supset L_m= 0$ such that $R_{L_k}/L_{k+1}$ is almost maximal $\forall k,\ 0\leq k\leq m-1$. Let $G$ be a torsion-free $R$-module of finite rank. Then $G$ has a pure-composition series
\[0=G_0\subseteq G_1\subseteq\dots\subseteq G_{m-1}\subseteq G_m\subseteq G_{m+1}=G\]
where $G_{j+1}/G_j$ is a finite direct sum of ideals of $R_{L_j}$, $\forall j,\ 0\leq j\leq m$.
\end{proposition}
\begin{proof} By Proposition~\ref{P:torsion-free1} $G$ has a pure-composition series
\[0=H_0\subset H_1\subset\dots\subset H_{n-1}\subset H_n=G\]
such that, $\forall k,\ 1\leq k\leq n,\ U_k=H_k/H_{k-1}$ is a uniserial torsion-free module and $U_k^{\sharp}\supseteq U_{k+1}^{\sharp},\ \forall k,\ 1\leq k\leq n-1$. Now, for each $j,\ 1\leq j\leq m$, let $k_j$ be the greatest index such that $L_j\subset U_{k_j}^{\sharp}$. We put $G_j=H_{k_j}$. Then $G_{j+1}/G_j$ is an $R_{L_j}$-module which is a direct sum of ideals by Lemma~\ref{L:ext2}. \end{proof}

\section{Valuation domains $R$ with $fr(R)<\infty$.}
\label{S:bound}

First we extend to every almost maximal valuation domain the methods used by Lady in \cite{Lad77} to study  torsion-free  modules over rank-one discrete valuation domains. So, except in Theorem~\ref{T:n=2}, we assume that $R$ is an almost maximal valuation ring. We put $K=Q/R$. For each $R$-module $M$, $\mathrm{d}(M)$ is the divisible submodule of $M$ which is the union of all divisible submodules and $M$ is said to be \textbf{reduced} if $\mathrm{d}(M)=0$. We denote by $\widehat{M}$  the \textbf{pure-injective hull} of $M$ (see \cite[chapter XIII]{FuSa01}). If $U$ is a uniserial module then $\widehat{U}\cong\widehat{R}\otimes_RU$ because $R$ is almost maximal.
Let $G$ be a torsion-free module of finite rank $r$. By Proposition~\ref{P:torsion-free} $G$ contains a submodule $B$ which is a direct sum of ideals and such that $G/B$ is a $Q$-vector space.  We put 
 \textbf{corank} $G=\mathrm{rank}\ G/B$. Now, it is easy to prove the following.

\begin{proposition}
\label{P:rankCo} Let $G$ be a torsion-free $R$-module of $\mathrm{rank}\ r$ and $\mathrm{corank}\ c$. Then:
\begin{enumerate}
\item $G$ contains a  pure direct sum  $B$ of ideals, of rank $r-c$, such that $G/B$ is a $Q$-vector space of dimension $c$.
\item $G$ contains a pure direct sum  $B'$ of ideals, of rank $r$ such that $G/B'$ is isomorphic to a quotient of $K^c$.
\end{enumerate} 
\end{proposition}

\bigskip 
An element of $Q\otimes_R\mathrm{Hom}_R(G,H)$ is called a \textbf{quasi-homomorphism} from $G$ to $H$, where $G$ and $H$ are $R$-modules. Let $\mathcal{C}_{{\rm ab}}$ be the category having  \textbf{weakly polyserial}  $R$-modules (i.e modules with composition series whose factors are uniserial) as objects and  quasi-homomorphisms as morphisms and let $\mathcal{C}$ be full subcategory of $\mathcal{C}_{{\rm ab}}$
 having  torsion-free $R$-modules of finite rank as objects. Then $\mathcal{C}_{{\rm ab}}$ is abelian by \cite[Lemma XII.1.1]{FuSa01}. If $G$ and $H$ are torsion-free of finite rank, then the quasi-homomorphisms from $G$ to $H$ can be identified with the $Q$-linear maps $\phi:Q\otimes_RG\rightarrow Q\otimes_RH$ such that $r\phi(G)\subseteq H$ for some $0\ne r\in R$. We say that $G$ and $H$ are \textbf{quasi-isomorphic}  if they are isomorphic objects of $\mathcal{C}$. A torsion-free module of finite rank is said to be \textbf{strongly indecomposable} if it is an indecomposable object of $\mathcal{C}$.    
\begin{lemma}
 \label{L:quasi} Let  $\phi:Q\otimes_RG\rightarrow Q\otimes_RH$ be a $Q$-linear map. Then $\phi$ is a quasi-homomorphism if and only if $(\widehat{R}\otimes_R\phi)(d(\widehat{R}\otimes_RG))\subset d(\widehat{R}\otimes_RH)$. 
 \end{lemma}
 \begin{proof} Assume that $\phi$ is a quasi-homomorphism. There exists $0\ne r\in R$ such that $r\phi(G)\subseteq H$. It successively follows that $r(\widehat{R}\otimes_R\phi)(\widehat{R}\otimes_RG)\subseteq \widehat{R}\otimes_RH$ and $(\widehat{R}\otimes_R\phi)(d(\widehat{R}\otimes_RG))\subseteq d(\widehat{R}\otimes_RH)$.
 
 Conversely, let $B$ be a finite direct sum of ideals which satisfies that $G/B$ is a $Q$-vector space. There exists a free submodule $F$ of $Q\otimes_RG$ such that $B\subseteq F$. So, $\exists 0\ne r\in R$ such that $r\phi(B)\subseteq r\phi(F)\subseteq H$. Since $\widehat{R}\otimes_RG=(\widehat{R}\otimes_RB)\oplus (d(\widehat{R}\otimes_RG))$, it follows that $r(\widehat{R}\otimes_R\phi)(\widehat{R}\otimes_RG)\subseteq (\widehat{R}\otimes_RH)$. We deduce that $r\phi(G)\subseteq (\widehat{R}\otimes_RH)\cap (Q\otimes_RH)= H$.  \end{proof}

\begin{proposition}
\label{P:Endo} Let $G$ be a torsion-free $R$-module of $\mathrm{rank}\ r$ and $\mathrm{corank}\ c$.
\begin{enumerate}
\item If $G$ has no  summand isomorphic to an ideal, then $\mathrm{End}(G)$ can be embedded in the ring of $c\times c$ matrices over $\widehat{Q}$. In particular if $c=1$, $\mathrm{End}(G)$ is a commutative integral domain.
\item If $G$ is reduced, then $\mathrm{End}(G)$ can be embedded in the ring of $(r-c)\times (r-c)$ matrices over $\widehat{R}$. In particular if $c=r-1$, $\mathrm{End}(G)$ is a commutative integral domain.
\end{enumerate} 
\end{proposition}
\begin{proof} See the proof of \cite[Theorem 3.1]{Lad77}. \end{proof}

\bigskip 
 
 In the sequel we assume that $n=c_R(0)<\infty$.
 So, there are $n-1$ units $\pi_2,\dots,\pi_n$ in $\widehat{R}\setminus R$ such that $1,\pi_2,\dots,\pi_n$ is a basis of $\widehat{Q}$ over $Q$. By \cite[Theorem XV.6.3]{FuSa01} there exists an indecomposable torsion-free $R$-module $E$ with $\mathrm{rank}\ n$ and $\mathrm{corank}\ 1$. We can define $E$ in the following way: if $(e_k)_{2\leq k\leq n}$ is the canonical basis of $\widehat{R}^{n-1}$, if $e_1=\sum_{k=2}^{k=n}\pi_ke_k$ and $V$ is the $Q$-vector subspace of $\widehat{Q}^{n-1}$ generated by $(e_k)_{1\leq k\leq n}$, then $E=V\cap\widehat{R}^{n-1}$. Then a basis element for $\mathrm{d}(\widehat{R}\otimes E)$ can be written $u_1+\pi_2u_2+\dots+\pi_nu_n$, where $u_1,\dots,u_n\in E$. Since $E$ is indecomposable it follows that $u_1,\dots,u_n$ is a basis for $Q\otimes E\cong V$. 
\begin{theorem}
\label{T:struc} Let $G$ be a torsion-free $R$-module of $\mathrm{rank}\ r$ and $\mathrm{corank}\ c$. Then the following assertions hold:
\begin{enumerate}
\item The reduced quotient of $G$ is isomorphic to a pure submodule of $\widehat{B}$ where $B$ is a direct sum of $(r-c)$ ideals.
\item $G$ is the direct sum of ideals of $R$  with a quasi-homomorphic image of $E^c$.
\end{enumerate} 
\end{theorem}
\begin{proof} $(1)$ can be shown as the implication $(1)\Rightarrow (2)$ of \cite[Theorem 4.1]{Lad77} and
$(2)$ as the implication $(1)\Rightarrow (3)$ of \cite[Theorem 4.1]{Lad77}. \end{proof}

\begin{corollary}
\label{C:rank} Let $G$ be a torsion-free $R$-module of $\mathrm{rank}\ r$ and $\mathrm{corank}\ c$. Then:
\begin{enumerate}
\item If $G$ has no summand isomorphic to an ideal, then $r\leq nc$.
\item If $G$ is reduced, then $nc\leq (n-1)r$.
\end{enumerate} 
\end{corollary}
\begin{proof} This corollary is a consequence of Theorem~\ref{T:struc} and can be shown as  \cite[Corollary 4.2]{Lad77}. \end{proof}

\begin{theorem}
\label{T:n=2} Let $R$ be a valuation domain such that $d_R(0)=2$. Then $fr(R)=2$. Moreover $fr^o(R)=1$ if $c_R(0)=2$ and $fr^o(R)=2$ if $c_R(0)=1$.
\end{theorem}
\begin{proof} First suppose that $c_R(0)=2$. So, $R$ is almost maximal and $fr^o(R)=1$. Let $G$ be an indecomposable torsion-free module with $\mathrm{rank}\ r$ and $\mathrm{corank}\ c$ which is not isomorphic to $Q$  and to an ideal. Then $G$ is reduced and has no summand isomorphic to an ideal of $R$. From Corollary~\ref{C:rank} we deduce that $r=2c$. By Theorem~\ref{T:struc} $G$ is isomorphic to a pure submodule of $\widehat{B}$ where $B$ is a direct sum of $c$ ideals. Since $\mathrm{rank}\ \widehat{B}=2c$ it follows that $G\cong\widehat{B}$. So, $c=1$ and $G\cong\widehat{A}$ for a non-zero ideal $A$. 

If $c_R(0)=1$ let $L$ be the non-zero prime ideal such that $c_R(L)=d_{R/L}(0)=2$. Then $fr(R/L)=2$. Since $R_L$ is maximal it follows that $fr(R)=fr^o(R)=2$ by \cite[Lemma 9 and Lemma 4]{Vam90}. \end{proof}

\begin{lemma}
\label{L:subE} Every proper subobject of $E$ in $\mathcal{C}$ is a direct sum of ideals.
\end{lemma}
\begin{proof} Let $G$ be a proper object of $E$ in $\mathcal{C}$ and let $H$ be the pure submodule of $E$ such that $H/G$ is the torsion submodule of $E/G$. Since $E$ is indecomposable, $E$ has no summand isomorphic to a direct sum of ideals. So, $\mathrm{corank}\ E/H=1$ and $\mathrm{corank}\ H=0$. As $\mathrm{corank}\ H\geq \mathrm{corank}\ G$ we get that $G$ is a direct sum of ideals. \end{proof}

\begin{proposition}
\label{P:Eproj} $E$ is an indecomposable projective object of $\mathcal{C}$.
\end{proposition}
\begin{proof} Let $\phi:H\rightarrow E$ be a quasi-epimorphism where $H$ is a torsion-free module of finite rank. Suppose that $H=F\oplus G$ where $F$ is a direct sum of ideals. By Lemma~\ref{L:subE}, $\phi(G)$ is quasi-isomorphic to $E$. So, we may assume that $H$ has no summand isomorphic to an ideal. By Theorem~\ref{T:struc} there is a quasi-epimorphism $\psi:E^c\rightarrow H$ where $c=\mathrm{corank}\ H$. It is sufficient to see that $\phi\circ\psi$ is a split epimorphism in $\mathcal{C}$. But by Proposition~\ref{P:Endo}(1), $Q\otimes\mathrm{End}(E)$ is a subfield of $\widehat{Q}$, so every quasi-homomorphism $E\rightarrow E$ is either a quasi-isomorphism or trivial and the splitting follows immediately. \end{proof} 

\bigskip

In the sequel, $Q\otimes_R\mathrm{Hom}_R(R\oplus E,M)$ is denoted by $\ddot{M}$ for each $R$-module $M$ and the ring $Q\otimes_R\mathrm{End}_R(R\oplus E)$ by $\Lambda$.

\begin{theorem}
\label{T:func} The foncteur $Q\otimes_R\mathrm{Hom}_R(R\oplus E,\ )$ is an exact fully faitful functor from $\mathcal{C}$ into $\mathrm{mod-}\Lambda$, the category of finitely generated right $\Lambda$-modules.
\end{theorem}
\begin{proof} By Theorem~\ref{T:struc} and Proposition~\ref{P:Eproj}, $R\oplus E$ is a progenerator of $\mathcal{C}$. For each finite rank torsion-free $R$-module $H$, the natural map $Q\otimes_R\mathrm{Hom}_R(R\oplus E,H)\rightarrow\mathrm{Hom}_{\Lambda}(\ddot{R}\oplus\ddot{E},\ddot{H})$ is an isomorphism because $\Lambda=\ddot{R}\oplus\ddot{E}$. Thus $Q\otimes_R\mathrm{Hom}_R(F,H)\rightarrow\mathrm{Hom}_{\Lambda}(\ddot{F},\ddot{H})$ is an isomorphism if $F$ is a  summand of a finite direct sum of modules isomorphic to $R\oplus E$. Let $G$ be a finite rank torsion-free $R$-module.
We may assume that $G$ has no summand isomorphic to an ideal of $R$. By Proposition~\ref{P:Eproj} and Lemma~\ref{L:subE}, there is an exact sequence $0\rightarrow R^{nc-r}\rightarrow E^c\rightarrow G\rightarrow 0$ in $\mathcal{C}$. Since both functors are left exact, we get that $Q\otimes_R\mathrm{Hom}_R(G,H)\cong\mathrm{Hom}_{\Lambda}(\ddot{G},\ddot{H})$. \end{proof}

\begin{lemma}
\label{L:submod} If $M$ is a right $\Lambda$-module and $M\subseteq\ddot{G}$ for some finite rank torsion-free $R$-module $G$, then $M\cong\ddot{H}$ for some torsion-free $R$-module $H$.
\end{lemma}
\begin{proof} See the proof of \cite[Lemma 5.2]{Lad77}. \end{proof}

\begin{proposition}
\label{P:hered} The ring $\Lambda$ is a hereditary Artinian $Q$-algebra such that \\
$(\mathrm{rad}\ \Lambda)^2=0$. There are two simple right $\Lambda$-modules, $\ddot{R}$ which is projective and $\ddot{K}$ which is injective.
\end{proposition}
\begin{proof} See the proof of \cite[Proposition 5.3]{Lad77}. \end{proof}

\begin{proposition}
\label{P:cover}$\ddot{Q}$ is an injective hull for $\ddot{R}$ and $\ddot{E}$ is a projective cover for $\ddot{K}$.
\end{proposition}
\begin{proof} See the proof of \cite[Proposition 5.4]{Lad77}. \end{proof}

\begin{theorem}
\label{T:subcat} The image of $\mathcal{C}$ under the functor $Q\otimes_R\mathrm{Hom}_R(R\oplus E,\ )$ is the full subcategory of $\mathrm{mod-}\ \Lambda$ consisting of modules with no summand isomorphic to $\ddot{K}$.
\end{theorem}
\begin{proof} See the proof of \cite[Theorem 5.5]{Lad77}. \end{proof}

\bigskip Let $M$ be a finitely generated (i.e.,finite length) right $\Lambda$-module. We define \textbf{rank} $M$ to be the number of factors in a composition series for $M$ isomorphic to $\ddot{R}$ and \textbf{corank} $M$ to be the number of composition factors isomorphic to $\ddot{K}$.

\begin{proposition}
\label{P:rank} The foncteur $Q\otimes_R\mathrm{Hom}_R(R\oplus E,\ )$ preserves rank and corank.
\end{proposition}
\begin{proof} See the proof of \cite[Proposition 5.6]{Lad77}. \end{proof}

\bigskip

We now consider the functors $\mathrm{\mathbf{D}}=\mathrm{Hom}_R(\ ,Q)$ and $\mathrm{\mathbf{Tr}}=\mathrm{Ext}_{\Lambda}(\ ,\Lambda)$ which take right $\Lambda$-modules to left $\Lambda$-modules and conversely. It is well known that $\mathrm{\mathbf{D}}$ is an exact contravariant lenght preserving functor taking projectives to injectives and conversely, and that $\mathrm{\mathbf{D}}^2$ is the identity for finitely generated $\Lambda$-modules. Since $\Lambda$ is hereditary, $\mathrm{\mathbf{Tr}}$ is right exact and $\mathrm{\mathbf{Tr}}^2\ M\cong M$ if $M$ has no projective summand, $\mathrm{\mathbf{Tr}}\ M=0$ if $M$ is projective. We consider the Coxeter functors $\mathrm{C}^+=\mathrm{\mathbf{DTr}}$ and $\mathrm{C}^-=\mathrm{\mathbf{TrD}}$. Thus $\mathrm{C}^+:\mathrm{mod-}\Lambda\rightarrow\mathrm{mod-}\Lambda$ is left exact and $\mathrm{C}^-:\mathrm{mod-}\Lambda\rightarrow\mathrm{mod-}\Lambda$ is right exact. If $M$ has no projective (respectively injective) summand, it is easy to check that $M$ is indecomposable if and only if $\mathrm{C}^+\ M$ (respectively $\mathrm{C}^-\ M$) is indecomposable.

\begin{proposition}
\label{P:Coxeter} Let $M$ be a right $\Lambda$-module with $\mathrm{rank}\ r$ and $\mathrm{corank}\ c$.
\begin{enumerate}
\item If $M$ has no projective summand, then $\mathrm{corank}\ \mathrm{C}^+\ M=(n-1)c-r$ and $\mathrm{rank}\ \mathrm{C}^+\ M=nc-r$.
\item If $M$ has no injective summand, then $\mathrm{rank}\ \mathrm{C}^-\ M=(n-1)r-nc$ and $\mathrm{corank}\ \mathrm{C}^-\ M=r-c$.
\end{enumerate} 
\end{proposition}
\begin{proof} See the proof of \cite[Proposition 5.7]{Lad77}. \end{proof}

\begin{theorem} The following assertions hold:
\begin{enumerate}

\item If $n=3$, then, up to quasi-homomorphism, the strongly indecomposable torsion-free $R$-modules are $R,\ Q,\ E,\ \widehat{R}$ and an  $R$-module with $\mathrm{rank}\ 2$ and $\mathrm{corank}\ 1$ (corresponding to $\mathrm{C}^+\ \ddot{Q}=\mathrm{C}^-\ \ddot{R}$).
\item If $n\geq 4$, there are strongly indecomposable torsion-free $R$-modules with arbitrarily large rank.
\end{enumerate} 
\end{theorem}
\begin{proof} 
We show $(1)$ and $(2)$ as Lady in the proof of \cite[Theorem 5.11]{Lad77} by using Proposition~\ref{P:Coxeter} and \cite[Theorem 2]{Mul74}. \end{proof}

\section{Some other results and  open questions.}

Let $G$ be a finite rank torsion-free module over an almost maximal valuation domain $R$. A \textbf{splitting field} for $G$ is a subfield $Q'$ of $\widehat{Q}$ containing $Q$ such that $(Q'\cap\widehat{R})\otimes_RG$ is a completely decomposable $(Q'\cap\widehat{R})$-module (i.e. a direct sum of rank one modules). If $Q'$ is a splitting field for $G$, $G$ is called $Q'$\textbf{-decomposable}. By \cite[Theorem 7]{Men83}, each finite rank torsion-free module $G$ has a unique minimal splitting field $Q'$ and $[Q':Q]<\infty$. So, Lady's results on splitting fields of torsion-free modules of finite rank over rank one discrete valuation domains can be extended to almost maximal valuation domains by replacing $\widehat{Q}$ by $Q'$ in the previous section and by taking $\mathcal{C}$ to be the category whose objects are the finite rank torsion-free $Q'$-decomposable modules.

Now, $R$ is a valuation domain which is not necessarily almost maximal. We say that an $R$-module $G$ is \textbf{strongly flat} if it is an extension of a free module by a divisible torsion-free module (see \cite{BaSa04}). By \cite[Lemma V.1.1 and Proposition V.1.2]{FuSa85} $\widetilde{Q}$ is a splitting field for each finite rank strongly flat module. So, each finite rank strongly flat module $G$ has a unique minimal splitting field $Q'\subseteq\widetilde{Q}$ and $[Q':Q]<\infty$. We also can extended Lady's results. In particular:
\begin{theorem}
Let $R$ be a valuation domain. We consider the following conditions:
\begin{enumerate}
\item $\exists l\geq 1$ such that each finite rank strongly flat module is a direct sum of modules of rank at most $l$;
\item $c_R(0)\leq 3$.
\end{enumerate} 
Then $(1)\Rightarrow (2)$. If $c_R(0)\leq 2$ then $(1)$ holds and $l=c_R(0)$.

\end{theorem}

If $c_R(0)=3$, it is possible that the proof of \cite[Theorem 2.6]{ArDu93} should be generalized. 

\begin{proof} $(1)\Rightarrow (2)$. We show that $c_R(0)\leq\infty$ in the same way that the condition (a) of \cite[Theorem 3]{Vam90} is proven. Then, as above, we use Lady's methods to get $c_R(0)\leq 3$.

If $c_R(0)=2$ we do as in the proof of Theorem~\ref{T:n=2}. \end{proof}

\bigskip
\textbf{Some open questions:}
\begin{enumerate}
\item Does henselian valuation domains with finite total defect, which are not strongly discrete, exist?
\item For a valuation domain $R$, does the condition $d_R(0)=3$ imply $fr(R)=3\ \mathrm{or}\ <\infty$? It is possible that the proof of \cite[Theorem 2.6]{ArDu93} should be generalized. 
\item For a valuation domain $R$ with finite total defect, does the condition  $c_R(L_i)=p,\ \forall i,\ 1\leq i\leq m$, where $p=2\ \mathrm{or}\ 3$ and $(L_i)_{1\leq i\leq m}$ is the family of prime ideals defined in Theorem~\ref{T:main}, imply $fr(R)<\infty$ for some  $m>1$? (If $m>1$, by \cite[Theorem 10]{Vam90}, $R$ is Henselian and $p$ is the characteristic of its residue field.)
\end{enumerate} 
 


\end{document}